\documentclass[12pt]{amsart}

\usepackage{amssymb}
\usepackage[initials]{amsrefs}

\makeatletter
\@namedef{subjclassname@2010}{%
  \textup{2010} Mathematics Subject Classification}
\makeatother

\newtheorem{theorem}{Theorem}[section]
\newtheorem{corollary}[theorem]{Corollary}
\newtheorem{lemma}[theorem]{Lemma}
\newtheorem{proposition}[theorem]{Proposition}

\theoremstyle{definition}

\newtheorem*{remark}{Remark}

\numberwithin{equation}{section}

\newcommand{\F}{\mathbb{F}}
\DeclareMathOperator{\Tr}{Tr}
\DeclareMathOperator{\ord}{ord}
\DeclareMathOperator{\lcm}{lcm}
\newcommand{\Q}{\mathcal{Q}}

\newcommand{\X}{\mathcal{X}}
\newcommand{\Y}{\mathcal{Y}}
\newcommand{\Z}{\mathcal{Z}}
\newcommand{\N}{\mathcal{N}}
\newcommand{\M}{\mathcal{M}}

\begin{document}


\baselineskip=17pt


\title{The trace of 2-primitive elements of finite fields (amended version)}

\author[S.D.~Cohen]{Stephen D. Cohen } \thanks{The first author is Emeritus Professor of Number Theory, University of Glasgow}

\address{6 Bracken Road, Portlethen, Aberdeen AB12 4TA, Scotland, UK}
\email{Stephen.Cohen@glasgow.ac.uk}
\author[G. Kapetanakis]{Giorgos Kapetanakis}
\address{Department of Mathematics and Applied Mathematics\\ University of Crete\\ Voutes Campus, 70013 Heraklion, Greece}
\email{gnkapet@gmail.com}

\date{\today}

\begin{abstract}
Let $q$ be a prime power and $n, r$ integers such that $r\mid q^n-1$. An element of $\F_{q^n}$ of multiplicative order $(q^n-1)/r$ is called \emph{$r$-primitive}. For any odd prime power $q$, we show that there exists a $2$-primitive element of $\F_{q^n}$ with arbitrarily prescribed $\F_q$-trace when $n\geq 3$.  Also we explicitly describe the values that the trace of such elements may have when $n=2$. A feature of this amended version is the reduction of the discussion to extensions  of prime degree $n$.
\end{abstract}

\subjclass[2010]{Primary 11T30; Secondary 11T06}

\keywords{Primitive elements, high order elements, trace function}

\maketitle

\section{Introduction}
Let  $q$ be a power of the prime $p$ and $n\geq 2$ an integer. We denote by $\F_q$ the finite field of $q$ elements, by $\F_p$ its prime subfield and by $\F_{q^n}$ its extension of degree $n$.
It is well-known that the multiplicative group $\F_{q^n}^*$ is cyclic: a generator is called a  \emph{primitive} element. The theoretical importance of primitive elements is complemented by their numerous applications in practical areas such as cryptography.

Elements of $\F_{q^n}^*$ of high order, without necessarily being primitive, are important as in several applications one may use instead of primitive elements. Hence many authors have considered their effective construction \cites{gao99,martinezreis16,popovych13}. We call \emph{$r$-primitive} an element of order $(q^n-1)/r$, where $r\mid q^n-1$. In this sense, primitive elements are exactly the $1$-primitive elements. Recently, the existence of $2$-primitive elements that also possess other desirable properties was considered \cite{kapetanakisreis18}, while clearly such elements exist only in finite fields of odd characteristic.

We denote by $\Tr$ the trace function $\F_{q^n}\to\F_q$, that is
\[
\Tr(\xi) := \sum_{i=0}^{n-1} \xi^{q^i},\ \xi\in\F_{q^n} .
\]
One property that has been extensively studied is that of the prescribed trace, while there are numerous results in the literature about elements combining the above with other desirable properties like prescribed norm, primitivity, normality, etc. In particular, the possible traces of primitive elements has been explicitly described \cite{cohen90}.
\begin{theorem}\label{thm:prim_trace}
Let $q$ be a prime power, $n$ an integer and $\beta\in\F_q$. Unless $(n,\beta) = (2,0)$ or $(n,q)=(3,4)$, there exists a primitive $\xi\in\F_{q^n}$ with $\Tr(\xi)=\beta$.
\end{theorem}
On the contrary, not so much is known about the possible traces of $r$-primitive elements. In this direction, as a direct consequence of \cite{cohen05a}*{Theorem~1.1}, one obtains the following:
\begin{theorem}
\label{thm:cohen2}
Let $q$ be a prime power and $\beta\in\F_q$. If $n$ and $r$ are such that $n > 4 \cdot (1 + \log_q(9.8 \cdot r^{3/4}) )$ and $r\mid q^n-1$, then there exists some $r$-primitive $\xi\in\F_{q^n}$ such that $\Tr(\xi)=\beta$.
\end{theorem}
It is worth mentioning that the above is a consequence of a more general result, where a number of rational expressions of an $r$-primitive element have prescribed traces.
In this work, we study the trace of $2$-primitive elements and prove the following analogue to Theorem~\ref{thm:prim_trace}:
\begin{theorem}\label{thm:main}
Let $q$ be an odd prime power.
\begin{enumerate}
\item Let $\beta\in\F_q$. There exists some $2$-primitive element $\xi$ of $\F_{q^n}$ such that $\Tr(\xi)=\beta$, for any $n\geq 3$.
\item Let $\beta\in\F_q^*$. There exists some $2$-primitive element $\xi$ of $\F_{q^2}$ such that $\Tr(\xi)=\beta$, unless $q=3,5,7,9,11,13$ or $31$. The possible choices of $\beta$ for those prime powers are listed in Table ~$\ref{tab:n=2_traces}$.
\end{enumerate}
\end{theorem}
Notice that the exclusion of even prime powers $q$ in Theorem~\ref{thm:main} is essential for the existence of $2$-primitive elements, so from now on we assume that $q$ is odd.

In order to prove Theorem~\ref{thm:main}, first we distinguish between {\em odd} and {\em even} pairs $(q,n)$ according as $\frac{q^n-1}{2}$  is odd or even, respectively.  In particular, $(q,n)$ is odd if $q \equiv 3 \pmod 4$ and $n$ is odd, whereas $(q,n)$ is even if either $q \equiv 1 \pmod 4$ or $n$ is even.  We remark that, if $(q,n)$ is even and $n$ is odd, then either $q$ is a power of a prime $p \equiv 1  \pmod 4$ or $q$ is a square. 

In Section~\ref{sec:oddpairs}, we reduce the problem of prescribing the trace of $2$-primitive elements to prescribing the trace of primitive elements when $(q,n)$ is odd. In Section~\ref{sec:sums}, we provide some background material, which is used in Section~\ref{sec:conditions} in order to prove conditions for the existence of $2$-primitive elements with prescribed trace when $(q,n)$ is even. As a by-product of the method we obtain an exact expression for the number of squares in $\F_{q^n}$ whose trace has a prescribed value in $\F_q$ (Proposition ~\ref{m=1}).  Finally, in Section~\ref{sec:existence}, we use the theory developed in the preceding section in order to complete the proof of Theorem~\ref{thm:main}.

This article is an amended version (with corrections) of \cite{cohenkapetanakis20}.

%
%
\section{Reduction to extensions of prime degree}\label{sec:primedegree}
Here we show how the proofs of Theorems \ref{thm:prim_trace} and \ref{thm:main} can be deduced from the case in which $n$ is a prime.  This would have reduced the working for the former in both \cite{cohen90} and \cite{cohenpresern05} and, of course facilitates the proof of the latter in the present article.  The principle would be  of potential value in the investigation of the existence of $r$-primitive elements with prescribed trace for $r>2$ too. In this section, given $d|n$ we denote by $\Tr_{n/d}$ the trace function from $\F_{q^n}$ to $\F_{q^d}$.
\begin{lemma}\label{induction} Suppose that either of Theorem~\ref{thm:prim_trace} or \ref{thm:main} have been established for all prime values of $n$ and  prime powers $q$ (with $q$ odd in the latter case). Then these theorems hold for arbitrary values of $n$.
\begin{proof} Let $t$ denote the total  number of primes (including their mutiplicity) in the prime decomposition of $n$.  The proof is by induction on $t$ with the case of prime $t$ corresponding to the stated assumption.

For the induction step write $n=lm$, where $l$ is any prime dividing $n$ and $1<m<n$.  Given $\beta \in \F_q$, by elementary linear algebra, it is evident that thre are $q^{l-1}$ elements in $\F_{q^l}$ whose trace in $\F_q$ is $\beta$.  Hence, in every case, even if $\beta=0$, we can choose {\em nonzero} $\alpha \in \F_{q^l}$  such that $\Tr_{l/1}(\alpha)=\beta$.  Next, for this nonzero  element $\alpha$  apply the induction hypothesis to the extension $\F_{q^n}=\F_{q^{lm}}$ over $\F_{q^l}$.  Now,  the exceptional cases in the theorems can only be relevant if  $q^l$  is a prime (which it is not) or, in Theorem \ref{thm:main}, when $q^l=9$  and $m=2$.   Moreover, this latter situation  can occur only if   $q=3, l=2$  and $n=4$ (so that $t=2$).   We deduce that, with $r=1$ or $2$ respectively,  there exists an $r$-primitive element  $\xi \in \F_{q^n}$ such that $\Tr_{n/l}(\xi)=\alpha$, with the possible exception of  the case when $r=2$, $t=2$,  $q=3$,  $n=4, l=2$ and $\alpha= \pm 1$.   But, from Table  \ref{tab:n=2_traces}, each of $\pm 1 \in  \F_{81}$ is  $2$-primitive and   $\Tr_{4/2}(\pm 1)={\mp 1}$, respectively, so that there is no exception even in this case.   Finally, $\Tr(\xi)=\Tr_{l/1}(\Tr_{n/l}(\xi))=\Tr_{l/1}(\alpha)=\beta$ and the lemma follows by induction.
\end{proof}
\end{lemma}
\section{Odd pairs}\label{sec:oddpairs}
\begin{lemma}\label{oddqn}
Suppose $(q,n)$  is odd. Then $\xi \in \F_{q^n}$ is $2$-primitive if and only if $-\xi$ is primitive.
\end{lemma}
\begin{proof}
We have that $\frac{q^n-1}{2}$ is odd, so $\xi$ is $2$-primitive if and only if $\xi$ is both $\frac{q^n-1}{2}$-free and a square in   $\F_{q^n}$ . Now $(-1)^{(q^n-1)/2}=-1$, thus $-1$ is a nonsquare in $\F_{q^n}$.   Hence $\xi \in \F_{q^n}$ is a nonsquare  if and only if $-\xi$ is a square.
Moreover, $\xi$ is $\frac{q^n-1}{2}$-free if and only if $-\xi$ is $\frac{q^n-1}{2}$-free.  The result follows.
\end{proof}
It follows from Lemma~\ref{oddqn} that when $(q,n)$ is odd, the number of $2$-primitive elements in $\F_{q^n}$ is the same as the number of primitive elements (namely $\phi(q^n-1)$).
In the situation of Lemma~\ref{oddqn} we can deduce the existence theorem for $2$-primitive elements from Theorem~\ref{thm:prim_trace}.
\begin{theorem}\label{oddthm}
Suppose $(q,n)$ is odd.  Then, given arbitrary $\beta \in \F_q$, there exists a $2$-primitive element of $\F_{q^n}$ with trace $\beta$.
\end{theorem}
\begin{proof}
From Theorem~\ref{thm:prim_trace}, there exists a primitive element $\xi$ of $\F_{q^n}$ with trace $-\beta$.  By Lemma~\ref{oddqn}, $-\xi$ is $2$-primitive and $\Tr(-\xi)=-\Tr(\xi)=\beta$.
\end{proof}
\begin{remark}
Since Theorem~\ref{thm:prim_trace} was established in \cite{cohenpresern05}, without recourse to direct verification for any pair $(q,n)$, then the same can be said for Theorem~\ref{oddthm}.               
\end{remark}
From now on we assume (as we may) that $(q,n)$ is even, in which case $(-1)^{(q^n-1)/2}=1$ and so $-1$ is a square in $\F_{q^n}$.   Thus $\xi \in \F_{q^n}$ is $2$-primitive if and only if $-\xi$ is  $2$-primitive.  In this situation, a $2$-primitive element can be viewed simply as the square $\xi^2$ of a primitive element $\xi $.  Hence, our problem is to confirm that there exists a primitive element $\xi \in \F_{q^n}$ for which $\Tr(\xi^2) =\beta$.  Observe that, if $\xi$ is primitive, then both $(\pm \xi)^2$ yield the same $2$-primitive element $\xi^2$.  In particular, it is clear that the the total number of $2$-primitive elements in $\F_{q^n}$ is $\frac{\phi(q^n-1)}{2}$.
\section{Character sums}\label{sec:sums}
We begin by introducing the notion of freeness. Let $m\mid q^n-1$. An element $\xi \in \F_{q^n}^*$ is \emph{$m$-free} if $\xi = \zeta^d$ for some $d\mid m$ and $\zeta\in\F_{q^n}^*$ implies $d=1$. It is clear that primitive elements are exactly those that are $q_0$-free, where $q_0$ is the square-free part of $q^n-1$. It is also evident that there is some relation between $m$-freeness and multiplicative order.
\begin{lemma}[\cite{huczynskamullenpanariothomson13}*{Proposition~5.3}]\label{lemma:m-free}
If $m\mid q^n-1$ then $\xi\in\F_{q^n}^*$ is $m$-free if and only if $\gcd\left( m,\frac{q^n-1}{\ord(\xi)} \right)=1$.
\end{lemma}

Throughout this work, a \emph{multiplicative character} is a multiplicative character of $\F_{q^n}^*$; in particular,  we denote by $\chi_1$ the trivial multiplicative character and by $\eta$ the quadratic character, i.e., for $\xi\in\F_{q^n}^*$,
\[
\chi_1(\xi)=1 \text{ and } \eta(\xi) = \begin{cases} 1, & \text{if $\xi$ is a square,} \\ -1,&\text{otherwise.} \end{cases}
\]
Vinogradov's formula yields an expression of the characteristic function of $m$-free elements in terms of multiplicative characters, namely:
\[
\omega_m(\xi) := \theta(m) \sum_{d\mid m} \frac{\mu(d)}{\phi(d)} \sum_{\ord(\chi_d) = d} \chi_d(\xi) ,
\]
where $\mu$ stands for the M\"{o}bius function, $\phi$ for the Euler function, $\theta(m) := \phi(m)/m$ and the inner sum suns through multiplicative characters $\chi_d$  of order $d$. Furthermore, a direct consequence of the orthogonality relations is that the characteristic function for the elements of $\F_{q^n}$ that are $k$-th powers, where $k\mid q^n-1$, can be written as
\[
w_k (\xi) := \frac{1}{k} \sum_{d\mid k} \sum_{\ord(\chi_d)=d} \chi_d(\xi) .
\]

Also, we will encounter additive characters of both $\F_{q^n}$ and $\F_q$.
Let $\psi$ be the canonical additive character of $\F_q$, that is $\psi (g) = \exp (2\pi i\Tr_0(g)/p)$, where $\Tr_0$ stands for the {\em absolute trace} of $g\in\F_q$, i.e., its trace over $\F_p$, the prime subfield of $\F_q$.
Then an arbitrary additive character of $\F_q$ has the action which takes $g\in\F_q$ onto $\psi (ug)$ and thereby, as $u$ varies, we obtain all the $q$ additive characters of $\F_q$,
whose set we will denote by $\widehat{\F_q}$.  For the trivial character, take $u=0$.
With the help of the orthogonality relations, it is straightforward to check that the characteristic function for elements of $\F_{q^n}$, with trace $\beta$, can be expressed as
\[
t_\beta (\xi) := \frac{1}{q} \sum_{u\in\F_q} \bar\psi(u\beta) \tilde\psi(u\xi) ,
\]
where, $\bar\psi$ stands for the inverse of $\psi$ and $\tilde\psi$ stands for the \emph{lift} of $\psi$ to an additive character of $\F_{q^n}$, i.e., for every $\xi\in\F_{q^n}$, we have that $\tilde\psi(\xi) = \psi(\Tr(\xi))$. In particular, $\tilde\psi$ is the canonical character of $\F_{q^n}$.
We conclude this section by presenting the character sum estimates we will use.

\begin{proposition}\label{propo:hybrid}
Let $\chi_d$ be a multiplicative character of order $d$, $u\in\F_q$ and $r\mid q^n-1$. Set
\[
A := \sum_{\xi\in\F_{q^n}} \chi_d(\xi) \tilde\psi(u\xi^r) .
\]
\begin{enumerate}
\item If $d=1$ and $u=0$, then $A = q^n-1$.
\item If $d=1$ and $u\neq 0$, then $|A| \leq (r-1) q^{n/2}+1$.
\item If $d\neq 1$ and $u=0$, then $A = 0$.
\item If $d\neq 1$ and $u\neq 0$, then $|A|\leq r q^{n/2}$.
\end{enumerate}
\end{proposition}
\begin{proof}
When $d=1$, since $\chi_1(0)=0$, we have that $A= \sum_{\xi \in \F_{q^n}} \tilde\psi(u\xi^r) -1$. From this, we immediately obtain the first item, while the second item is derived from \cite{schmidt76}*{Theorem~2E}.

The third item is a consequence of the orthogonality relations and the last item is implied by \cite{schmidt76}*{Theorem~2G}.
\end{proof}
When $r$ takes the specific value 2  and $d=1$,  we can evaluate the quantity $A$ appearing in Propostion \ref{propo:hybrid} precisely because it is associated with the quadratic Gauss sum.  Since we  only require the result when $(q,n)$ is even we assume (as we may) that if $n$ is odd and the  characteristic of $\F_q$ is a prime $p \equiv 3 \pmod 4$, then $q$ is a square. 
 We remark that, when $n$ is  even the every $u \in \F_q$ becomes a square in $\F_{q^n}$, whereas, when $n$ is odd, then a nonsquare $u \in \F_q$ remains a nonsquare in $\F_{q^n}$. 
 First  we note a standard fact about the quadratic Gauss  sum over over $\F_q$  (\cite{hua82}, Section 7.5, Theorem 5.4, \cite{lidlniederreiter97}, Section 5.2).  
\begin{lemma}\label{gauss}
 Let  $u \in \F_{q^n}$ and set  $$g_n(u)=\sum _{\xi \in \F_{q^n}}\psi(u\xi^2)= \sum_{\xi \in \F_{q^n}}\chi_2(\xi)\psi(u\xi),$$ where $\chi_2$ denotes the quadratic character and $\psi$ the canonical additive character on $\F_{q^n}$. Then  $$g_n(u)= \chi_2(u)g_n(1).$$
  \end{lemma}
By Lemma  ~\ref{gauss},  the following is a consequence of ~\cite{lidlniederreiter97}, Theorem 5.15 or \cite{hua82}, Section 7.5, Theorem 5.5.
\begin{lemma}\label{lem:A}
In the situation of Proposition $ \ref{propo:hybrid}$ take  $r=2$ and $d=1$ so that  $A=\sum_{\xi \in \F_{q^n}^*}\tilde \psi(u\xi^2)=g_n(u)-1.$  Suppose also that, if $n$ is odd and $p  \equiv 3 \pmod 4$, then $q$ is a square.  If $n$ is even then
\[ A= \varepsilon_1 q^{n/2}-1,\] where 
\begin{equation}\label{eps1}
 \varepsilon_1= \begin{cases}\ \ 1, & \text{if } q \equiv 3 \  (\mathrm{mod}\ 4) \text{ and } n \equiv 2  (\mathrm{mod} \ 4),\\
  -1, &  \text{if } q \equiv 1\  (\mathrm{mod}\ 4) \text{ or } 4\mid n.
\end{cases} 
\end{equation}
On the other hand, if $n$ is odd then
\[ A= \chi_2(u)\varepsilon_2 q^{n/2}-1, \] 
where $\chi_2$  denotes the quadratic character in $\F_q$ and
\begin{equation}\label{eps2}
 \varepsilon_2= \begin{cases}\ \ 1, & \text{if }  p \equiv 1\   (\mathrm{mod}\ 4) \text{ and }   q \text{ is a nonsquare},\\
 &    \text{or }   p \equiv 3 \  (\mathrm{mod}\ 4) \text{ and } q \text{ is a square but not a $4$th power},\\
 -1, &   \text{if }  p \equiv  1 \  (\mathrm{mod}\ 4)   \text{ and }   q \text{ is a square},\\
 &    \text{or }   p \equiv 3 \  (\mathrm{mod}\ 4)  \text{ and } q \text{ is  a $4$th power}. 
\end{cases} 
\end{equation}

\end{lemma}
Finally, the following is an improvement of the main result of \cite{katz89}, in the case $n=2$,
see \cite{cohen10}*{Lemma~3.3}.
\begin{lemma}\label{lem:katz}
Let $\theta\in\F_{q^2}$ be such that $\F_{q^2} = \F_q(\theta)$ and $\chi$ a non-trivial character. Set
\[ B := \sum_{\alpha\in\F_q} \chi (\theta+\alpha) . \]
\begin{enumerate}
\item If $\ord(\chi)\nmid q+1$, then $|B| = \sqrt{q}$.
\item If $\ord(\chi)\mid q+1$, then $B = -1$.
\end{enumerate}
\end{lemma}
\section{Conditions for even pairs}\label{sec:conditions}
%
%
In this section we provide able conditions for the existence of $2$-primitive elements with prescribed trace, when the pair $(q,n)$ is even.

Let $W(t)$ be the number of the square-free divisors of $t$. The following provides a bound for this number.
\begin{lemma}\label{lem:w(r)}
Let $t, a$ be positive integers and let $p_1, \ldots, p_j$ be the distinct prime divisors of $t$ such that $p_i\le 2^a$. Then $W(t)\le c_{t, a}t^{1/a}$, where
\[ c_{t, a}=\frac{2^j}{(p_1\cdots p_j)^{1/a}}. \]
In particular,
$d_t:=c_{t, 8}<4514.7$ for every $t$.
\end{lemma}
\begin{proof}
The statement is an immediate generalization of \cite{cohenhuczynska03}*{Lemma~3.3} and can be proved using multiplicativity. The bound
 for $d_t$ can be easily computed.
\end{proof}
\begin{remark}
Given $t$, the number $d_t$ is easily computed and in most cases (especially if $t$ is not too large) the actual value of $d_t$ is remarkably smaller than the generic bound.
\end{remark}
Recall that $2$-primitive elements are exactly the squares of primitive elements. In other words, we are looking for a primitive element the trace of whose square is fixed to some $\beta\in\F_q$. With that in mind, following the analysis of Section~\ref{sec:sums}, we define the following
\[
\N_\beta(m) := \sum_{\xi\in\F_{q^n}} \omega_m(\xi) t_\beta (\xi^2) ,
\]
where $m\mid q^n-1$. In particular, our aim is to prove that  $\N_\beta(q_0)\neq 0$  (where we recall that $q_0$ stands for the square-free part of $q^n-1$) and note that, in fact, since $(q,n)$ is even, $\N_\beta(q_0)$ counts twice the number of $2$-primitive elements with trace $\beta$. Next, we compute:
\begin{align}\
\frac{\N_\beta(m)}{\theta(m)} & =  \frac 1q \sum_{\xi\in\F_{q^n}} \sum_{d \mid m} \frac{\mu(d)}{\phi(d)} \sum_{\ord(\chi_d)=d} \chi_d(\xi) \sum_{u\in\F_q} \bar\psi(u\beta) \tilde\psi(u\xi^2) \nonumber \\
 & = \frac 1q  \sum_{d \mid m} \frac{\mu(d)}{\phi(d)} \sum_{\ord(\chi)=d} \sum_{u\in\F_q}\bar\psi(u\beta) \X_u (\chi_d) , \label{eq:N1} 
\end{align}
where
\begin{equation}\label{eq:X}
\X_u (\chi_d) :=  \sum_{\xi\in\F_{q^n}} \chi_d(\xi) \tilde\psi(u\xi^2) .
\end{equation}

  To help with the general argument we proceed with the precise evaluation of  $\N_\beta:=\N_\beta(1)$.  Observe that $\N_\beta=2\M_\beta$, where $\M_\beta$  is the number of  number of nonzer {\em squares} $\xi \in  \F_{q^n}$ with $\Tr(\xi)=\beta$.
  
  \begin{proposition}\label{m=1}
  Assume $(q,n$) is even.  If $n$ is even then,  for $\beta \in \F_q$, 
 \begin{equation}\label{eq:M1} 
 \M_\beta=\begin{cases}\frac{1}{2}\left(q^{n-1}-1 +\varepsilon_1(q-1)q^{\frac{n}{2}-1}\right), & \text{if } \beta=0,\\
  \frac{1}{2}\left(q^{n-1}- \varepsilon_1 q^{\frac{n}{2}-1}\right), & \text{if } \beta \neq 0,
   \end{cases},
   \end{equation}
   where $\varepsilon_1$  is given by\eqref{eps1}.
 
On the other hand, if $n$ is odd then
\begin{equation}\label{eq:M2}
\M_\beta= \begin{cases}\frac{1}{2}\left(q^{n-1}-1\right),  & \text{if } \beta=0,\\
\frac{1}{2}\left(q^{n-1} +\chi_2(\beta)q^{\frac{n-1}{2}}\right), & \text{if } \beta \neq 0,
\end{cases}
\end{equation}
where $\lambda$ denotes the quadratic character in $\F_q$.
    \end{proposition}
    
    \begin{proof}
    Begin by observing that, given a fixed nonsquare  $c \in \F_q$, the  elements of $\F_q^*$ can be written as a disjoint union $\{u^2; u \in \F_q^*\}\cup\{cu^2; u \in \F_q^*\}$,
 where each member of $\F_q^*$ appears twice. 
 
 Next, from  \eqref{eq:N1}
    \[q\N_\beta= q^n-1 + \sum_{u\in\F_q^*}\X_u,  \]
    where $\X_u(\chi_1)$ is abbreviated to $\X_u$.
  Take the case in which $n$ is even.   Then every $u \in \F_q^*$ is a square in $\F_{q^n}$ so that, by Lemma \ref{lem:A}, $\X_u=\varepsilon_1 q^{n/2}-1$.  Hence,
  \[q\N_0=q^n-1+(q-1)(q^{n/2}-1)=q^n-q+(q-1)q^{n/2},\]
(where $\varepsilon_1$ is given by \eqref{eps1}).   \eqref{eq:M1} follows in this case.    But, if $\beta \neq 0$ , then
  \[q\N_\beta= q^n-1+\frac{1}{2}(\varepsilon_1 q^{n/2}-1)\left( \sum_{u \in \F_q^*}\bar\psi(\beta u^2)+\bar\psi(c\beta u^2)\right)=q^n-1-(\varepsilon_1 q^{n/2}-1), \]
    by Lemma \ref{gauss}, since $c$ is a nonsquare.  Again    \eqref{eq:M1} follows.
    
    Now, suppose, $n$ is odd so that nonsquares in $\F_q$ remain nonsquare in $\F_{q^n}$.  This time, by Lemma \ref{gauss},
       \[q\N_0= q^n-1 +\frac{1}{2}\sum_{u \in \F_q} (\X_{u^2}(\chi_1)+\X_{cu^2}(\chi_1) )=q^n-1-(q-1)= q^n-q  \]
    and this case of  \eqref{eq:M2} follows.  Finally suppose $\beta \neq 0$. We have
   \begin{align}
    q\N_{\beta} -(q^n-1)  = & \frac{1}{2}\left(\sum_{u\in \F_q^*} \bar\psi(\beta u^2)\X_{u^2}+\bar\psi(c\beta u^2)\X_{cu^2}\right) \nonumber\\
 = & \frac{1}{2}\left(\sum_{u\in \F_q^*} \bar\psi(\beta u^2)\X_1+\bar\psi(c\beta u^2)\X_c\right)  \label{psiX}  \\
= & \frac{1}{2}\{(\chi_2(\beta)\varepsilon_2 q^{1/2}-1)(\varepsilon_2 q^{n/2}-1)+ \nonumber \\ & (\chi_2(c\beta)\varepsilon_2 q^{1/2}-1)(\chi_2(c)\varepsilon_2 q^{n/2}-1)\} \nonumber \\
= & \chi_2(\beta)  q^{(n+1)/2}  +1, \nonumber
   \end{align}
   where,  at \eqref{psiX},  Lemma \ref{lem:A} with $n=1$   and $\varepsilon_2$ given by \eqref{eps2} is applied to the sums  over $u \in \F_q^*$ in addition to the sums over
    $\xi \in \F_{q^n}^*$ and we note that $\varepsilon_2^2=1$.
   This yields \eqref{eq:M2} more generally.
    \end{proof}
    We remark that, whether $n$ is even or odd, the displayed values of $\M_\beta$ are consistent with the fact that the number of nonzero squares in $\F_{q^n}$ is $(q^n-1)/2$.
  
  We proceed to  consider the contribution of the terms  on the right side of \eqref{eq:N1} with $d>1$: call this quantity $R_\beta(m)$. The argument echoes that of the proof of Lemma \ref{m=1} without the precision of the latter.
  
    By Proposition \ref{propo:hybrid}(3) we can suppose $u\neq0$.    Observe that
  \begin{align*}
  \sum_{u \in \F_q^*}\bar\psi(u\beta)\X_u(\chi_d)& =\frac{1}{2}\left(\sum_{u \in \F_q^*}\bar\psi(u^2\beta)\X_{u^2}(\chi_d) +\sum_{u \in \F_q^*}\bar\psi(cu^2\beta)\X_{cu^2}(\chi_d)\right) \\
   &=\frac{1}{2}\left(\sum_{u \in \F_q^*} \bar\chi_d(u)\bar\psi(\beta u^2)\X_1(\chi_d) +\sum_{u \in \F_q^*}\bar \chi_d(u)\bar\psi(c\beta u^2)\X_c(\chi_d)\right) \\    
  \end{align*}

 Thus,  again  writing $\F_q^*$ as a disjoint union of squares and nonsquares (each counted twice), we have  
\begin{equation}\label{C}
R_\beta(m)= \frac{1}{2q}\sum_{ \substack{d|m \\ d>1}}\frac{\mu(d)}{\phi(d)}\sum_{\ord (\chi_d)=d}\sum_{u \in \F_q^*}\big(\bar{\psi}(u^2\beta)\bar\chi_d(u)\X_1(\chi_d)+\bar{\psi}(u^2c\beta)\bar\chi_d(u)\X_{c}(\chi_d)\big),
\end{equation}
%
To proceed  we distinguish between the cases   $\beta\neq 0$ and  $\beta=0$.
\subsection{The case  $\beta\neq 0$ and $n\geq 2$}
In this situation (\ref{C}) can be rewritten as follows.
\begin{lemma}\label{D}
Assume $(q,n)$ is even and $\beta(\neq  0) \in \F_q$.  Then
 \begin{equation}\label{E}
 R_\beta(m)=\frac{1}{2q}\sum_{ \substack {d\mid m \\ d>1}}\frac{\mu(d)}{\phi(d)}\sum_{\ord (\chi_d)=d}\big(\overline{X_{\beta}(\chi_d)}\X_1(\chi_d)+\overline{X_{c\beta}(\chi_d)}\X_c(\chi_d)\big),
   \end{equation}
 where $X_{\beta}(\chi_d)$ (with $\chi_d$ restricted to $\F_q$) is the sum $\sum_{u \in \F_q} \chi_d(u)\psi(u^2 \beta)$ (i.e., the sum $\X_{\beta}(\chi_d)$ over $\F_q$ rather than $\F_{q^n}$).
 \end{lemma}

Next, we present a lower bound for  ${\mathcal N}_\beta(m)$,     to a condition for it to be positive.  Throughout we suppose $\varepsilon_1$ is defined by (\ref{eps1}) when $n$ is even with $\varepsilon_1 =0$ if $n$ is odd.  Also set $\varepsilon_0 =1$, if $n$ is odd and $0$, if $n$ is even.

\begin{theorem}\label{mainbound}
\label{H}
 Assume $(q,n)$ is even, where $q$ is an odd prime power and $n \geq2$. Let  $\beta \in \F_q^*$ and  $m$ be an even divisor of $q^n-1$ with $m_Q$ be the product of those primes in $m$ which divide $Q= \frac{q^n-1}{q-1}$.  Then
  \begin{equation}\label{I}
  {\mathcal N}_\beta(m) \geq  \theta(m)q^{\frac{n-1}{2}} \left\{q^{\frac{n-1}{2}} - 4W(m)+2W(m_Q)+1\right\}.
  \end{equation}
\end{theorem}
\begin{proof}
From Proposition~\ref{propo:hybrid}, for any $b \in \F_q^*$ and multiplicative character $\chi$, $|\X_b(\chi)| \leq 2 q^{\frac{n}{2}}$ and $|X_\beta(\chi)| \leq 2q^{\frac{1}{2}}$ .
 Indeed when $d\mid m_Q$ then $\chi_d$  restricted to  $\F_q$ is the trivial character and  $|X_\beta(\chi_d)| \leq q^{\frac{1}{2}}+1$.  Indeed, we can be more precise about the latter bound (i.e., when $d|m_q$). 
 For,  whether $n$ is even or odd, by \eqref{eps2},  $X_\beta(\chi_1)$ and $X_{c\beta}(\chi_1)$ take the two  real values $\pm q^{1/2} -1$ in either order. Thus , one of $|\big(\overline{X_{\beta}(\chi_d)}\X_1(\chi_d)|$  and $|\big(\overline{X_{c\beta}(\chi_d)}\X_c(\chi_d)|$ is bounded by $2(q^{1/2} -1)q^{n/2}$ and the other by  $2(q^{1/2} +1)q^{n/2}$. So their sum is bounded absolutely by $4q^{(n+1)/2}$  
 Thus, by  Lemma~\ref{D}, 
\begin{multline*} \mathcal{ N}_\beta(m)/\{\theta(m)q^{\frac{n-1}{2}}\}  \geq  \\q^{\frac{n-1}{2}} -\varepsilon_0 +\varepsilon_1/q^{1/2}- 4(W(m)-W(m_Q)) -2(W(m_Q)-1).
\end{multline*}
 and the  result  follows.
   \end{proof}
%

%
%
%
\subsection{The case   $\beta=0$ and $n\geq 3$}
Next we suppose that $n\geq 3$ and $\beta=0$. Note that, for $q>3$ and $n=2$, any $2$-primitive element cannot be zero-traced, i.e., it is essential to assume that $n\geq 3$ in this case.

Now (\ref{C}) does not have a Gauss sum factor. We show that, to ensure that ${\mathcal N}_0(q^n-1)$ is positive it suffices to show
that ${\mathcal N}_0(Q)$ is positive.
\begin{lemma}
\label{L}
Suppose $\xi \in \F_{q^n}$ is $Q$-free. 
Then there exists $c \in \F_q$ with $c\xi \in \F_{q^n}$ primitive. 
If, further, $\Tr(\xi^2) =0$, then $\Tr((c\xi)^2)=0$.
\end{lemma}
\begin{proof}
 It is possible that $q-1$ and $Q$ have a common prime factor (or factors), namely prime factors of $n$.  Express $q-1$ as a product $LM$, where $L$ and $M$ are coprime, such that $\xi$ is $QL$-free  and $\xi$ is an $m$-th power in $\F_{q^n}$ for each prime $m$ dividing $M$ (so $m\nmid QL$).  Hence, if $\gamma$ is a primitive element of $\F_{q^n}$, then $\xi=\gamma^{M_0t}$, where $t$ and $Q$ are coprime and $M_0$ is such that its square-free part is identical with the square-free part of $M$.  Define $g=\gamma^Q$, a primitive element of $\F_q$, and set $c=g^L=\gamma^{QL}$. Thus
 $c\xi=\gamma^{QLt+M_0t}$ is $QLM=q^n-1$-free.

 If actually $\Tr(\xi^2) =0$, then $\Tr((c\xi)^2)=\Tr(c^2\xi^2) =0$ since $c^2 \in \F_q$.
\end{proof}
\begin{lemma}
\label{lemma:C}
 Assume $(q,n)$ is even with $n\geq 3$ and that $m\mid Q$.
 \begin{equation}\label{M}
R_0(m)=\frac{q-1}{2q}\sum_{ \substack{d\mid m \\ d>1}}\frac{\mu(d)}{\phi(d)}\sum_{\ord (\chi_d)=d}(\X_1(\chi_d)+\X_{c}(\chi_d)).
\end{equation}
\end{lemma}
\begin{proof}
The above is an immediate consequence of \eqref{C}, after considering the fact that $\chi_d$ is trivial on $\F_q$ for every $d\mid Q$.
\end{proof}
\begin{theorem}\label{N}
Assume $(q,n)$ is even with $n\geq 3$. Suppose that $m\mid Q$. Then
 \begin{equation}\label{O}
   {\mathcal N}_0(m) \geq \theta(m)q^{\frac{n}{2}-1}\left\{q^{\frac{n}{2}}-2W(m)(q-1) \right\}.
\end{equation}
Consequently, if
\begin{equation} \label{O*}
q^{\frac{n}{2}}> 2W(Q)(q-1),
\end{equation}
 then ${\mathcal N}_0(q^n-1)>0$.
\end{theorem}
\begin{proof}
Lemma~\ref{lemma:C}, combined with Proposition~\ref{propo:hybrid} and Proposition \ref{m=1}, yields 
\[   {\mathcal N}_0(m) \geq \theta(m)q^{\frac{n}{2}-1}\left\{q^{\frac{n}{2}}-q^{1-n/2}+\varepsilon_1(q-1)-2(W(m)-1)(q-1) \right\}.\]
and \eqref{O} easily follows. Now, assume that \eqref{O*} holds. By \eqref{O}, there exists some $Q$-free $\zeta\in\F_{q^n}$ with $\Tr(\zeta^2)=0$ and from Lemma~\ref{L} this implies the existence of a primitive $\xi\in\F_{q^n}$ such that $\Tr(\xi^2)=0$.
\end{proof}
%
%
%
%
\subsection{The sieve}
Our next aim is to relax the conditions of the preceding subsections. Towards this end, we employ the Cohen-Huczynska  sieving technique, \cite{cohenhuczynska03}. For any divisor $m$ of $q^n-1$ in expressions such as ${\mathcal N}_\beta(m)$ we freely interchange between using $m$ or its radical.
\begin{proposition}[Sieving inequality]\label{prop:siev0}
Assume $(q,n)$ is even. Let $m\mid q_0$ (the square-free part of $q^n-1$) and $\beta\in\F_q$. In addition, let $\{r_1,\ldots,r_s\}$ be a set of divisors of $m$ such that $\gcd(r_i,r_j)=r_0$ for every $i\neq j$ and $\lcm(r_1,\ldots ,r_s)=m$. Then
\[
\N_\beta(m) \geq \sum_{i=1}^s \N_\beta(r_i) - (s-1)\N_\beta(r_0) .
\]
\end{proposition}
\begin{proof}
For any $l\mid q_0$, let $S_l$ be the set of $l$-free elements the trace of whose square is equal to $\beta$; thus $|S_l| =
\N_\beta(l)$.

We will use induction on $s$. The result is trivial for $s=1$. For $s=2$ notice that $S_{r_1} \cup S_{r_2} \subseteq S_{r_0}$ and that $S_{r_1} \cap S_{r_2} = S_m$. The result follows after considering the cardinalities of those sets.

Next, assume that our hypothesis holds for some $s\geq 2$. We shall prove our result for $s+1$. Set $r:=\lcm(r_1,\ldots ,r_s)$ and apply the case  $s=2$  on $\{r,r_{s+1}\}$. The result follows from the induction hypothesis.
\end{proof}
First suppose $\beta \neq 0$.  Let the radical   of $q^n-1$  be expressed as $kp_1\ldots p_s$, where $p_1,\ldots,p_s$ are distinct primes and
 $s \geq 0$ and define $\delta=1 - \sum_{i=1}^s \frac{1}{p_i}$, with $\delta =1$, if $s=0$. Suppose further that $p_i\mid Q$ for $i=1, \ldots,r$ and $p_i\nmid Q$ for $i=r+1,\ldots,s$.  Set $\delta_Q= 1- \sum_{i=1}^r\frac{1}{p_i}$.
\begin{theorem}
\label{T}
Assume $(q,n)$ is even with $n\geq2$.  Suppose $\beta\neq 0$. Define $\delta, \delta_Q$ as above and assume that $\delta$ is positive. Then
\begin{multline}
\label{S}
{\mathcal N}_\beta(q^n-1)\geq
 \delta\theta(k)q^{\frac{n-1}{2}}\bigg\{ q^{\frac{n-1}{2}}  \\ -4\left(\frac{s-1}{\delta}+2\right)W(k) + 2\left(\frac{r-1+\delta_Q}{\delta}+1\right)W(k_Q)\bigg\}.
\end{multline}
Hence, if
\[
 q^{\frac{n-1}{2}}> \\ 4\left(\frac{s-1}{\delta}+2\right)W(k)-2\left(\frac{r-1+\delta_Q}{\delta}+1\right)W(k_Q),\]
then ${\mathcal N}_\beta(q^n-1) >0$.
\end{theorem}
\begin{proof}
Proposition~\ref{prop:siev0} implies that, for any $\beta \in  \F_q$,
\begin{align}\label{P}
{\mathcal N}_\beta(q^n-1) \geq & \sum_{i=1}^s {\mathcal N}_\beta(kp_i) -(s-1){\mathcal N}_\beta(k)\nonumber\\
\geq & \delta{\mathcal N}_\beta(k) -\sum_{i=1}^s \left|{\mathcal N}_\beta(kp_i)-\left (1- \frac{1}{p_i}\right){\mathcal N}_\beta(k)\right|
\end{align}

In (\ref{P}) use (\ref{I}) with $m=k$ as a lower bound. For the absolute value of the difference expressions  we distinguish between values two cases according as  $p_i\mid Q$ or not.
Suppose $p_i\nmid Q$.  Then
\begin{equation}\label{Q}
\left|{\mathcal N}_\beta(kp_i)-\left (1- \frac{1}{p_i}\right){\mathcal N}_\beta(k) \right|\leq 4\theta(k)\left(1-\frac{1}{p_i}\right) q^{\frac{n-1}{2}}W(k),
\end{equation}
since $W(kp_i)-W(k)=W(k)$.  On the other hand, if $p_i\mid Q$, we have the improved bound
\begin{equation}\label{R}
\left|{\mathcal N}_\beta(kp_i)-\left (1- \frac{1}{p_i}\right){\mathcal N}_\beta(k) \right|\leq  \theta(k)\left(1-\frac{1}{p_i}\right) q^{\frac{n-1}{2}}\left\{4W(k)-2W(k_Q)\right\},
\end{equation}
 using also the fact that $W(k_Qp_i)-W(k_Q)=W(k_Q)$.
By combining (\ref{I}), (\ref{P}), (\ref{Q}) and (\ref{R})  we deduce that  (\ref{S}) holds.
\end{proof}
%
%
%
Finally, suppose $\beta=0$ and $n\geq 3$. We use the sieve version of the criterion (\ref{O*}) to obtain a result that depends on writing $Q$ (rather than $q^n-1$) as $Q= kp_1\ldots p_s$.
\begin{theorem}
\label{Z}
Assume $(q,n)$ is even with $n\geq 3$.
  With the notation  $Q= kp_1\ldots p_s$, with $p_1, \ldots, p_s$ distinct primes dividing $Q$, set $\delta = 1- \sum_{i=1}^s\frac{1}{p_i}$. Assume that $\delta$ is positive.  Then
\[
{\mathcal N}_0(Q)>
 \delta\theta(k)q^{\frac{n}{2}}\left\{q^{\frac{n}{2}-1}-2\left(\frac{s-1}{\delta}+2\right)W(k)\right\}.
\]
Hence, if
\[
 q^{\frac{n}{2}-1}>2\left(\frac{s-1}{\delta}+2\right)W(k),
\]
then ${\mathcal N}_0(q^n-1) >0$.
\end{theorem}
\begin{proof}
The proof follows the same pattern as that of Theorem~\ref{T}, this time with the difference being that \eqref{O} substitutes for \eqref{I}.
\end{proof}
For a multiplicative character $\chi$ of
  $\F_{q^n}$ denote by $G_n(\chi)$ the Gauss sum $G_n(\chi)=\sum_{\xi \in \F_{q^n}}\chi(\xi)\psi(\xi)$, where $\psi$ is the canonical additive character.  In particular, $G_n(\chi_2)= g_n(1)$ as used in Lemma ~\ref{gauss}.    Indeed, by Lemma \ref{gauss} we have

\[ \sum_{\xi\in\F_{q^n}} \psi (b\xi^2) = \chi_2(b) G_n(\chi_2) . \]

In the case in which $q$ is prime and $n=1$, the following lemma is established in \cite{wenpeng02}*{Lemma~4}. Here, we prove it more generally.
\begin{lemma}
\label{A1}
Assume $(q,n)$ is even. Let $\chi$ be any non-trivial multiplicative character of $\F_{q^n}$. Then, for any $b\in \F_{q^n}$,
\begin{equation} \label{B1}
|\X_b(\chi)|^2 = (1+\chi(-1))q^n+\chi_2(b)G_n(\chi_2)C(\chi),
\end{equation}
where $C(\chi):= \sum_{\xi \in  \F_{q^n}} \chi(\xi)\chi_2(\xi^2-1)$. Thus $|C(\chi)|\leq 2q^{\frac{n}{2}}$.
\end{lemma}
\begin{proof}
Let $\psi$ be the canonical additive character of $\F_{q^n}$. We have that
\begin{align*}
|\X_b(\chi)|^2 & = \sum_{\xi\in\F_{q^n}^*} \chi(\xi)\psi(b\xi^2) \overline{\sum_{\zeta\in\F_{q^n}^*} \chi(\zeta)\psi(b\zeta^2)} \\
 & = \sum_{\xi\in\F_{q^n}^*}\sum_{\zeta\in\F_{q^n}^*} \chi \left( \frac{\xi}{\zeta} \right) \psi(b(\xi^2 - \zeta^2)) \\
 & = \sum_{\xi\in\F_{q^n}^*} \sum_{\zeta\in\F_{q^n}^*} \chi(\xi) \psi (b\zeta^2(\xi^2-1)) \\
 & = \sum_{\xi\in\F_{q^n}^*} \chi(\xi) \left[ \sum_{\zeta\in\F_{q^n}} \psi (b\zeta^2(\xi^2-1)) -1 \right] \\
 & = (1+\chi(-1)) q^n + \sum_{\substack{\xi\in\F_{q^n}^* \\ \xi\neq\pm 1}} \chi(\xi) \sum_{\zeta\in\F_{q^n}} \psi (b\zeta^2(\xi^2-1)) - \sum_{\xi\in\F_{q^n}^*} \chi(\xi) .
\end{align*}
The result now follows from Lemma~\ref{gauss}.
\end{proof}
As we know from Lemma \ref{lem:A}, when $(q,n)$ is even, we have $G_n(\chi_2) = \pm q^{\frac{n}{2}}$.  We proceed with the implications of Lemma~\ref{A1} when $n$ is odd; in particular it applies in the key case when $n=3$. In this situation, since $(q,n)$ is even, necessarily $q \equiv 1 \pmod 4$.

\begin{lemma}
\label{C1}
Assume $q \equiv 1 \pmod 4$ and $n$ is odd.  Let $\chi$ be a non-trivial multiplicative character of $\F_{q^n}$ and $c$ a nonsquare in $\F_q$.  Then
\[
|\X_1(\chi)| +|\X_c(\chi)| \leq 2 \sqrt{2} q^{\frac{n}{2}}.
\]
\end{lemma}
\begin{proof}
Since in this context $Q$ is odd then $c$ remains a nonsquare in $\F_{q^n}$.  Thus $\chi_2(c)=-1=-\chi_2(1)$.  Hence, from (\ref{B1}),
\begin{equation}\label{eq:ip_C1}
\big(|\X_1(\chi)| +|\X_c(\chi)|\big)^2 =2(1+\chi(-1))q^n + 2|\X_1(\chi)||\X_c(\chi)| .
\end{equation}
Additionally, in a similar manner as in the proof of Lemma~\ref{A1}, we have that
\begin{align*}
\X_1(\chi)\overline{\X_c(\chi)} & = \sum_{\xi\in\F_{q^n}^*} \sum_{\zeta\in\F_{q^n}^*} \chi \left( \frac\xi\zeta \right) \psi(\xi^2 - c\zeta^2) \\
 & = \sum_{\xi\in\F_{q^n}^*} \sum_{\zeta\in\F_{q^n}^*} \chi(\xi) \psi(\zeta^2(\xi^2-c)) \\
 & = \sum_{\xi\in\F_{q^n}^*} \chi(\xi) \left[ \sum_{\zeta\in\F_{q^n}} \psi(\zeta^2(\xi^2-c)) - 1 \right] \\
 & = \sum_{\xi\in\F_{q^n}^*} \chi(\xi) \sum_{\zeta\in\F_{q^n}} \psi(\zeta^2(\xi^2-c)) .
\end{align*}
Now Lemma~\ref{gauss} yields that
\[ \X_1(\chi)\overline{\X_c(\chi)} = G_n(\chi_2) \sum_{\xi\in\F_{q^n}^*} \chi(\xi) \chi_2(\xi^2-c) . \]
From the fact that $G_n(\chi_2) = \pm q^{n/2}$ and that the (absolute value of the) inner sum is bounded by $2q^{n/2}$, it follows that
\[ |\X_1(\chi)||\X_c(\chi)| =  |\X_1(\chi)\overline{\X_c(\chi)}| \leq 2q^n \]
and the result follows once we insert the above in \eqref{eq:ip_C1}.
\end{proof}
By applying Lemma~\ref{C1} to (\ref{M}) (instead of the bound $|\X_b(\chi)| \leq 2q^{\frac{3}{2}}$)  and extending this to the sieve result we obtain the following improvements to Theorems~\ref{T} and \ref{Z}.  
\begin{theorem}
\label{E1}
Assume $q\equiv 1 \pmod 4$ and $n$ is odd.  With  the notation of Theorem~$\ref{T}$, assume  $\beta \neq 0$ and $\delta >0$.  Suppose
\[
  q^{\frac{n-1}{2}} >2\sqrt{2}\left(\frac{s-1}{\delta}+2\right)W(k)-\sqrt{2}\left(\frac{r-1+\delta_Q}{\delta}+1\right)W(k_Q).
\]
Then $\mathcal{N}_\beta(q^n-1) >0$.
\end{theorem}
\begin{theorem}
\label{G1}
Assume $q\equiv 1 \pmod 4$ and $n$ is odd.  In the situation of and with the notation of  Theorem~$\ref{Z}$, assume  $\beta =0$ and $\delta >0$. Suppose that
\[
q^{\frac{n}{2}-1}>\sqrt{2}\left(\frac{s-1}{\delta}+2\right)W(k).
\]
Then $\mathcal{N}_0(q^n-1) >0$.
\end{theorem}

%
%
\section{Existence results for even pairs}\label{sec:existence}
In this section we use the theory of the previous section in order to complete the proof of Theorem~\ref{thm:main}. All the computations mentioned in this section were realized with \textsc{SageMath}; we comment that a mid-range modern computer can perform them within a few seconds.    By Lemma \ref{induction} we can assume $n$ is prime.   If  $q$ is odd, then $q \equiv 1  \pmod 4$.   We distinguish the cases, $n$ a prime exceeding $4$,  $n=3$ and $n=2$.
\subsection{The case  $n>4$, prime}
We assume that  $\beta\in\F_q$. We begin by employing the simplest condition for $\N_\beta(q_0)\neq 0$ to check, that is
\begin{equation}\label{eq:n>4_1}
q^{\frac{3n}{8} -1} > 4 \cdot 4514.7 ,
\end{equation}
which is a consequence of Theorem~\ref{H}, Theorem~\ref{N} and Lemma~\ref{lem:w(r)}. We verify that the above holds for $n\geq 27$ and $q\geq 3$, which means that the case $n> 27$ is settled. Then, we check the validity of \eqref{eq:n>4_1} for $5\leq n\leq 23$ and compile Table~\ref{tab:n>4_1}.

\begin{table}[h]
\begin{center}\footnotesize
\begin{tabular}{|c||c|c|c|c|c|c|c|} \hline
$n$ & 5  & 7    & 11 & 13 & 17 & 19 & 23  \\ \hline
$q\geq$ & 73259 & 419  & 25 & 13 & 7 & 5& 4 \\ \hline

\end{tabular}
\caption{Pairs $(q,n)$ that satisfy \eqref{eq:n>4_1}.\label{tab:n>4_1}}
\end{center}
\end{table}

One can check that there are exactly 3735 pairs $(n,q)$, where $q$ is an odd prime power and $(q,n)$ is even, not covered by Table~\ref{tab:n>4_1}. For those pairs, we check the condition
\[ q^{\frac{3n}{8} -1} > 4 d_{q^n-1} , \]
where we compute $d_{q^n-1}$ for each pair explicitly. A computation reveals that all but the 5 pairs
$(q,n) =  (5, 5), (9, 5), (13, 5), (25, 5)$ and $(37, 5)$ satisfy the latter. 

 Then, we focus on the case $\beta\neq 0$ and notice that all the above 5 pairs satisfy the condition of Theorem~\ref{H}, when all the mentioned quantities are explicitly computed.

Finally, we focus on the case $\beta=0$  and verify that all 5 pairs satisfy the condition of Theorem~\ref{N}, with $W(Q)$ replaced by its exact value. To sum up, we have proved the following.
\begin{proposition}
\label{prop:main_n>4}
Let $q$ be an odd prime power and $n>4$, such that $(q,n)$ is even. For any $\beta\in\F_q$, there exists some $x\in\F_{q^n}$ with multiplicative order $(q^n-1)/2$ and $\Tr(x)=\beta$.
\end{proposition}
\subsection{The case  $n=3$}
Here we assume that  $n=3$ and $\beta\in\F_q$, where $q \equiv 1 \pmod {4}$ is prime.  

We note that the strategy of the previous subsection can be applied here, but the required computer resources (time and memory) prevented us from pursuing this path. Instead, we favoured a more computationally efficient strategy, where all the mentioned computations were performed within a few seconds.

Let $t(q,n)$ denote the number of prime divisors of $q^n-1$, so that $W(q^n-1) = 2^{t(q,n)}$.
A combination of Theorem~\ref{E1}, Theorem~\ref{G1} and Lemma~\ref{lem:w(r)} yields the condition
\[ q\geq q_2 := (2\sqrt{2} \cdot 4514.7)^8 \simeq 7.07 \cdot 10^{32} , \]
that is the case $q\geq q_2$ is settled and, since any product of any 54 distinct primes is larger than $q_2^3-1$, as a quick computation reveals, the case in which  $t(q,3) \geq 54$ is settled.

Let $p(i)$ denote the $i$-th prime, for instance $p(2)=3$. Our next step is to focus on dealing with smaller values of $t(q,3)$. We fix $n=3$ and  employ the following algorithm that accepts $t_1 \leq t_2$ as input and performs the following steps:
\begin{description}
\item[Step 1] Find the largest $s\leq t_1$ such that $\delta := 1- \sum_{i=0}^{s-1} 1/p(t_1-i) > 0$.
\item[Step 2] Compute $q_1 := \left( 4\cdot 2^{t_2-s} \cdot \left( \frac{s-1}{\delta} + 2 \right)  \right)^{2/(n-2)}$.
\item[Step 3] Find the largest $c$ such that $p(1)\cdots p(c) \leq q_1^n-1$.
\item[Step 4] If $c\leq t_1$ return SUCCESS, otherwise return FAIL.
\end{description}
If the algorithm returns SUCCESS, then the case $t_1\leq t(q,3) \leq t_2$ is settled.

Let us now explain the validity of the above algorithm that is based on  Theorems~\ref{E1} and \ref{G1}. Take some $q$ such that $t_1\leq t(q,3) \leq t_2$ and write $q^3-1 = p_1^{s_1} \cdots p_{t(q,3)}^{s_{t(q,3)}}$, where the $p_i$'s are the distinct prime divisors of $q^3-1$ in ascending order. In particular, notice that $p(i) \leq p_i$.
\begin{itemize}
\item In \textbf{Step~1}, we determine the number $s$ of primes that we are going to sieve. In particular, we will use $\{ p_{t_1-s+1} ,\ldots ,p_{t_1}\}$ and compute a lower bound for $\delta$. Notice that, since the smaller the $\delta$ the weaker the conditions of Theorems~\ref{E1} and \ref{G1} become; thus  it suffices to consider this value.
\item The quantity $q_1$ in \textbf{Step~2} is, according to  Theorems~\ref{E1} and \ref{G1}, such that if $q\geq q_1$, then the desired result follows.
\item In \textbf{Step~3}, $c$ stands for the maximum number of prime divisors that a number smaller than $q_1^n-1$ can admit.
\item If the check in \textbf{Step~4} is successful, then $c\leq t_1\leq t(q,3)$, that is $q^3-1$ has more prime divisors than any number smaller than $q_1^3-1$ can have and the answer follows.
\end{itemize}

Then we apply the above algorithm.  It turns out to be successful for the pairs $(t_1,t_2) = (35,53)$, $(29,34)$, $(24,28)$, $(21,23)$, $(19,20)$, $(18,18)$ and $(17,17)$, hence the case $t(q,3) \geq 17$ is also settled.

We continue with the cases $8\leq t(q,3)\leq 16$. For those values of $t(q,3)$ the algorithm of the previous subsection fails, but we apply it nonetheless in order to compute the quantity $q_1$ that appears on the second step. However, this time we choose the number of sieving primes to be one smaller than the maximum possible, as this seems to yield stronger results. Now, if $q>q_1$ and $t_1\leq t(q,3)\leq t_2$, then we obtain the desired result. We apply the algorithm for the pairs $(t,t)$, with $8\leq t\leq 16$. First, we notice that we get a positive answer for $t=16$ in this occasion, i.e., the case $t(q,3)=16$ is also settled, but also conclude that, for $t\leq 15$, $q_1\leq 511{,}095$. Hence the case $q>511{,}095$ and  $8\leq t(q,3)\leq 15$ is settled.

Additionally, for $t(q,3)\leq 7$, Theorems~\ref{E1} and \ref{G1} imply that the case $q\geq (2\sqrt{2}\cdot 2^7)^2 = 131{,}072$ is settled.  Thus, in short, for our purposes it suffices to check the cases $3\leq q\leq 511{,}095$.
In this interval, there are exactly 4459 odd prime powers $q$, such that $q\equiv 1\pmod{4}$, that do not satisfy $q^{1/2} > 2\sqrt{2}W(q_0)$, with $q=511{,}033$ being the largest among them.

We begin with the case $\beta\neq 0$. A quick computation verifies that all the  aforementioned 4459  prime powers satisfy the condition of Theorem~\ref{E1}, given that all mentioned quantities are replaced by their exact values, without resorting to sieving, with the exception of $q=5$, $9$, $13$, $25$, $29$, $61$ and $121$. However, we successfully apply sieving for $q=29$, $61$ and $121$, with $\{67,13,7\}$, $\{ 97,13,5 \}$ and $\{ 37,19,7 \}$ as our set of sieving primes respectively. This concludes the case $\beta\neq 0$.

Finally, assume $\beta=0$. A quick computation verifies that most of the 4459  prime powers satisfy the condition of Theorem~\ref{E1} when all relevant primes are given  their exact values, without resorting to sieving. The 15 exceptional prime powers are $q=5$, $9$, $13$, $25$, $29$, $37$, $49$, $61$, $81$, $109$, $121$, $277$, $289$, $373$ and $1369$. However, 8 of them successfully admit sieving, as presented in Table~\ref{tab:n=3,succ}.

\begin{table}[h]\footnotesize
\begin{center}
\begin{tabular}{|l|l|l||l|l|l|} \hline
$q$ & Sieving primes & \# & $q$ & Sieving primes & \#\\ \hline \hline
$29$ & $\{67, 13\}$ & 2 &
$61$ & $\{97, 13, 3\}$ & 3 \\
$81$ & $\{73, 13\}$ & 2 &
$109$ & $\{571, 7\}$ & 2 \\
$277$ & $\{193, 19, 7\}$ & 3 &
$289$ & $\{307, 13, 7\}$ & 3 \\
$373$ & $\{73, 13\}$ & 2 &
$1369$ & $\{67, 43\}$ & 2 \\ \hline
\end{tabular}
\end{center}
\caption{Prime powers $q$ that satisfy Theorem~\ref{G1} for $n=3$ and their respective sieving primes.\label{tab:n=3,succ}}
\end{table}

Summing up, we have proved the following.
\begin{proposition}
\label{prop:main_n=3}
Let $q$ be a prime power such that $q\equiv 1\pmod{4}$. For any $\beta\in\F_q$, there exists some $x\in\F_{q^3}$ with multiplicative order $(q^n-1)/2$ and $\Tr(x)=\beta$, unless $q=5$, $9$, $13$ or $25$ and $\beta\in\F_q$ or $q=37$, $49$ or $121$ and $\beta=0$.
\end{proposition}
\subsection{The case  $n=2$}
For $n=2$ it is straightforward to check that a $2$-primitive element cannot be zero-traced for $q\geq 5$, i.e., we assume that $\beta\neq 0$.   From Theorem \ref{mainbound}, we have the sufficient criterion that $\N_\beta(q^2-1)$ is positive whenever
\[ q^{1/2} > 4W(q^2-1)-2W(q+1)-1.\]
Similarly, there is a corresponding sieving criterion derivable from Theorem \ref{T}.  The adoption, however, of a  strategy found  in \cite{cohen90} yields stronger results, so we retain  this approach here.
\begin{lemma}\label{lemma:basis}
For every $\beta\in\F_q^*$, there exist $\theta_1,\theta_2\in\F_{q^2}$, such that $\{\theta_1,\theta_2\}$ is an $\F_q$-basis of $\F_{q^2}$, $\Tr(\theta_1)=\beta$ and $\Tr(\theta_2)=0$.
\end{lemma}
\begin{proof}
The trace function is onto, hence there exists some $\theta_1\in\F_{q^2}$ such that $\Tr(\theta_1)=\beta$. Next, extend $\{\theta_1\}$ to an $\F_q$-basis of $\F_{q^2}$, say $\{\theta_1,\theta_2'\}$ and set $\theta_2 := \theta_2' - \frac{\Tr(\theta_2')}{\Tr(\theta_1)} \cdot \theta_1$. It is clear that $\{\theta_1,\theta_2\}$ satisfies the desired conditions.
\end{proof}
\begin{corollary}\label{cor:basis}
Let $\beta,\theta_1,\theta_2$ be as in Lemma~$\ref{lemma:basis}$. For every $\alpha\in\F_q$, we have that $\Tr(\theta_1+\alpha\theta_2)=\beta$.
\end{corollary}
Fix $\beta\in\F_q^*$ and let $\theta_1,\theta_2$ be as in Lemma~\ref{lemma:basis}. Since $\{\theta_1,\theta_2\}$ are $\F_q$-linearly independent, we have that $\theta_1/\theta_2\not\in\F_q$, that is $\F_{q^2}=\F_q(\theta_1/\theta_2)$. In addition, Corollary~\ref{cor:basis} implies that for every $\alpha\in\F_q$, $\Tr(\theta_1+\alpha\theta_2)=\beta$.

Write $q^2-1=2^\ell q_2$, where $q_2$ is odd, and notice that, since $q$ is odd, $4\mid q^2-1$, that is $\ell \geq 2$, while the fact that $\gcd(q-1,q+1)=2$ implies that $q_2=r_2s_2$ where $r_2$ and $s_2$ are the $2$-free parts of $q+1$ and $q-1$ respectively and they are co-prime. Also, set $q_2',r_2'$ and $s_2'$ as the square-free parts of $q_2,r_2$ and $s_2$ respectively.

Next, take $r\mid q_2'$ and set $\Q_r$ to be the number of $r$-free elements of the form $\theta_1+\alpha\theta_2$ for some $\alpha\in\F_q$, that are squares but not 4th powers. Following the analysis of Section~\ref{sec:sums}, we get that
\begin{align}
\Q_r  & = \sum_{x\in\F_q} \omega_{r}(\theta_1+x\theta_2) w_2(\theta_1+x\theta_2) (1-w_4(\theta_1+x\theta_2)) \nonumber \\
 & =  \sum_{x\in\F_q} \omega_{r}(\theta_1+x\theta_2) (w_2(\theta_1+x\theta_2) -w_4(\theta_1+x\theta_2)) , \label{eq:Qr}
\end{align}
given that, by definition, for all $\xi\in\F_{q^2}^*$, $w_2(\xi)w_4(\xi) = w_4(\xi)$. In addition, notice that, for all $\xi\in\F_{q^2}^*$, we have that
\begin{align}
w_2(\xi)-w_4(\xi) & = \frac{1}{2} \sum_{\delta\mid 2} \sum_{\ord(\chi_\delta)=\delta} \chi_\delta(\xi) - \frac{1}{4} \sum_{\delta\mid 4} \sum_{\ord(\chi_\delta)=\delta} \chi_\delta(\xi) \nonumber \\
 & = \frac{1}{2} \sum_{\delta\mid 4} \sum_{\ord(\chi_\delta) = \delta} \ell_{\delta} \chi_\delta(\xi) , \label{eq:w2-w4}
\end{align}
where,
\[ 
\ell_{\delta} := \begin{cases} 1/2 , & \text{if } \delta=1\text{ or }2 , \\ -1/2 , & \text{if } \delta=4 . \end{cases}
\]
Furthermore, Lemma~\ref{lemma:m-free} implies that an element is $q_2'$-free if and only if it is $2^i$-primitive for some $0\leq i\leq\ell$. It follows that $2$-primitive elements of $\F_{q^2}$ are the $q_2'$-free elements that are squares, but not 4th powers. In other words, it suffices to show that $\Q_{q_2'}\neq 0$, while it is clear that 
\[\Q_{q_2'}\neq 0 \Rightarrow \N_\beta(q^2-1) \neq 0 . \]
 
In \eqref{eq:Qr}, we replace $\omega_{r}$ by its expression and $w_2-w_4$ by its expression in \eqref{eq:w2-w4} and we get that
\begin{equation}\label{eq:Q1}
\frac{4\Q_r}{\theta(r)} = \left( \sum_{\substack{d\mid r \\ \delta\mid 4}} \frac{\mu(d)}{\phi(d)} 2 \ell_\delta \sum_{\substack{\ord(\chi_d)=d \\ \ord(\chi_\delta)= \delta}} \Y(\chi_d,\chi_\delta) \right)  = \left( \sum_{d\mid r}\frac{\mu(d)}{\phi(d)} \sum_{\ord(\chi_d)=d} \Z(\chi_d) \right),
\end{equation}
where
\[
\Y(\chi_d,\chi_\delta) := \sum_{\alpha\in\F_q} \psi_{d,\delta} (\theta_1+\alpha\theta_2) = \psi_{d,\delta}(\theta_2)\sum_{\alpha\in\F_q} \psi_{d,\delta} \left( \frac{\theta_1}{\theta_2} + \alpha \right)
\]
and
\[
\Z(\chi_d) := \Y(\chi_d,\chi_1) + \Y(\chi_d,\chi_2) - \Y(\chi_d,\eta_1) -\Y(\chi_d,\eta_2) ,
\]
where $\psi_{d,\delta}:=(\chi_d\chi_\delta)$ is the product of the corresponding characters, $\eta$ is the quadratic character and $\eta_1,\eta_2$ are the two multiplicative characters of order exactly $4$. Furthermore, since $d$ is odd and $\delta\mid 4$, it is clear that $\psi_{d,\delta}$ is trivial if and only if $d=\delta=1$. 

Recall that $\F_{q^2} = \F_q(\theta_1/\theta_1)$.
First, assume $q\equiv 1\pmod{4}$. Then $4\nmid q+1$. Hence Lemma~\ref{lem:katz} implies that
\begin{enumerate}
  \item for $\chi_1$, $|\Z(\chi_1)| \geq q-1-2\sqrt{q}$,
  \item for $1\neq \ord(\chi_d)\mid q+1$, $|\Z(\chi_d)| \leq 2+2\sqrt{q}$, 
  \item for $\ord(\chi_d)\nmid q+1$, $|\Z(\chi_d)| \leq 4\sqrt{q}$.
\end{enumerate}
Next,  assume that $q\equiv 3\pmod{4}$.  Then $4\mid q+1$ and Lemma~\ref{lem:katz} implies that
\begin{enumerate}
  \item for $\chi_1$, $|\Z(\chi_1)|\geq q-3$,
  \item for $1\neq \ord(\chi_d)\mid q+1$, $|\Z(\chi_d)| \leq 4$,
  \item for $\ord(\chi_d)\nmid q+1$, $|\Z(\chi_d)| \leq 4\sqrt{q}$.
\end{enumerate}
We insert the above in \eqref{eq:Qr} and arrive at  the following.
\begin{proposition}\label{prop:n=2(1)}
Let $q$, and $r$ be as above and let $r_1$ be the product of the prime divisors of $r$ that divide $q+1$.
\begin{enumerate}
\item If $q\equiv 1\pmod{4}$, then
\begin{equation} \label{eq:n=2_ip(1)}
\frac{4 \Q_r}{\theta(r)} \geq q+1 - 4W(r)\sqrt{q} + 2W(r_1) (\sqrt{q}-1);
\end{equation}
that is, if
\[ q+1 > 4\left( W(r)\sqrt{q} - W(r_1) \left( \frac{\sqrt{q}-1}{2} \right)\right) ,\]
then $\Q_r\neq 0$.
\item If $q\equiv 3\pmod{4}$, then
\begin{equation} \label{eq:n=2_ip(2)}
\frac{4\Q_r}{\theta(r)} \geq q+1 - 4W(r) \sqrt{q} + 4W(r_1)(\sqrt{q}-1);
\end{equation}
that is, if
\[ q+1 > 4(W(r)\sqrt{q} - W(r_1)(\sqrt{q}-1)) ,\]
then $\Q_r\neq 0$.
\end{enumerate} 
\end{proposition}
Our next aim is to relax the above conditions and, towards this end, we once more employ the Cohen-Huczynska sieving technique,  \cite{cohenhuczynska03}.
\begin{proposition}[Sieving inequality]\label{prop:siev1}
Let $r\mid q_2'$ and $\{r_1,\ldots,r_s\}$ a set of divisors of $r$ such that $\gcd(r_i,r_j)=r_0$ for every $i\neq j$ and $\lcm(r_1,\ldots ,r_s)=r$. Then
\[
\Q_{r} \geq \sum_{i=1}^s \Q_{r_i} - (s-1)\Q_{r_0} .
\]
\end{proposition}
\begin{proof}
This proof and that  of Proposition~\ref{prop:siev0} share the same pattern.
\end{proof}
Write $q_2'=kp_1\cdots p_s$, where $p_1$,\ldots,$p_s$ are distinct primes and $\varepsilon := 1-\sum_{i=1}^s 1/p_i$, with $\varepsilon=1$ when $s=0$. Further, suppose that $p_i\mid q+1$ for $i=1,\ldots,r$ and $p_i\nmid q+1$ for $i=r+1,\ldots,s$. Finally, set $\varepsilon' := 1-\sum_{i=1}^r 1/p_i$ and let $k_1$ be the part of $k$, that divides $q+1$.
\begin{theorem}\label{thm:siev4}
Let $q$ and $q_2'$ be as above. Additionally, let $\varepsilon$ and $\varepsilon'$ be as above and assume that $\varepsilon>0$.
\begin{enumerate}
\item If $q\equiv 1\pmod{4}$ and
\[ q+1 > 4\left[ W(k)\left( \frac{s-1}{\varepsilon} +2 \right) \sqrt{q} - W(k_1) \left( \frac{r-1+\varepsilon'}{\varepsilon} +1 \right) \left( \frac{\sqrt{q}-1}{2} \right)\right] ,\]
then $\Q_{q_2'}\neq 0$.
\item If $q\equiv 3\pmod{4}$ and
\[ q+1 > 4\left[  W(k)\left( \frac{s-1}{\varepsilon} +2 \right)\sqrt{q}- W(k_1) \left( \frac{r-1+\varepsilon'}{\varepsilon} +1 \right) (\sqrt{q}-1) \right]  ,\]
then $\Q_{q_2'}\neq 0$.
\end{enumerate} 
\end{theorem}
\begin{proof}
Proposition~\ref{prop:siev1} implies that
\[
\Q_{q_2'}\ \geq\ \sum_{i=1}^s \Q_{kp_i} - (s-1) \Q_{k} 
\ \geq\ \varepsilon \Q_{k} - \sum_{i=1}^s \left| \Q_{kp_i} - \left( 1-\frac{1}{p_i} \right) \Q_{k} \right| . \label{eq:siev_ip(0)}
\]
Notice that $\theta(kp_i) = \theta(k)(1-1/p_i)$. It follows from \eqref{eq:Q1} that
\begin{equation} \label{eq:siev_ip(1)}
\Q_{kp_i} - \left( 1-\frac{1}{p_i} \right) \Q_{k} =\frac{\theta(k)(p_i-1)}{4p_i} \sum_{d\mid k} \frac{\mu(dp_i)}{\phi(dp_i)} \sum_{\ord(\chi_{dp_i})=dp_i} \Z(\chi_{dp_i}) .
\end{equation}

First assume that $q\equiv 1\pmod{4}$. We repeat the arguments that led us to \eqref{eq:n=2_ip(1)} for \eqref{eq:siev_ip(1)}. If $i=1,\ldots,r$, i.e., $p_i\mid q+1$, then
\begin{multline*}
\left| \Q_{kp_i} - \left( 1-\frac{1}{p_i} \right) \Q_{k} \right| \leq \\ \theta(k) \left( 1-\frac{1}{p_i} \right) \Big[ 2\sqrt{q} (W(k)-W(k_1)) + (1+\sqrt{q}) W(k_1) \Big] ,
\end{multline*}
since $W(kp_i) = 2W(k)$ and $W(k_1p_i) = 2W(k_1)$. Similarly, if $i=r+1,\ldots,s$, i.e., $p_i\nmid q+1$, then
\[
\left| \Q_{kp_i} - \left( 1-\frac{1}{p_i} \right) \Q_{k}(\theta,\alpha) \right| \leq \theta(k) \left( 1-\frac{1}{p_i} \right) 2\sqrt{q} W(k).
\]
The combination of \eqref{eq:n=2_ip(1)}, \eqref{eq:siev_ip(0)}, \eqref{eq:siev_ip(1)} and the above bounds yields the desired result.

The $q\equiv 3\pmod{4}$ case follows similarly, but with \eqref{eq:n=2_ip(2)} in mind.
\end{proof}
We are now ready to proceed with our existence results.
We start with the simplest condition to check, which follows from Proposition~\ref{prop:n=2(1)} and the fact that $W(k)=W(q^2-1)/2$, namely
\[ \sqrt{q} \geq 2W(q^2-1) . \]
The above, with the help of Lemma~\ref{lem:w(r)}, implies that the case
\[ q \geq q_0 = (2\cdot 4514.7)^4 \simeq 6.65 \cdot 10^{15} \]
is settled. Next, a quick computation reveals that, if $q^2-1$ includes 14 or more prime numbers in its factorization, then $q\geq q_0$. In other words, if $t(q)$ stands for the number of prime factors of $q^2-1$, the case $t(q)\geq 14$ is settled.

Let $p(i)$ stand for the $i$-th prime (for example $p(2)=3$). Based on Theorem~\ref{thm:siev4}, we employ the following algorithm that takes $t_1\leq t_2$ as input and goes through the following steps:
%
%
\begin{description}
  \item[Step~1] Find the largest $s\leq t_1$ such that $\varepsilon_1:= 1-\sum_{i=0}^{s-1} \frac{1}{p(t_1-i)} >0$.
  \item[Step~2] Compute $q_1 := \left( 2\cdot 2^{t_2-s} \cdot \left( \frac{s-1}{\varepsilon_1} + 2 \right) \right)^2$.
  \item[Step~3] Find the largest $c$ such that $p(1)\cdots p(c) \leq q_1^2-1$.
  \item[Step~4] If $c\leq t_1$ return SUCCESS, otherwise return FAIL.
\end{description}
If the above returns SUCCESS, then the case $t_1\leq t(q)\leq t_2$ is settled.

Let us now explain the validity of the above algorithm. Assume that the returned value is SUCCESS for some $t_1\leq t_2$. Take some $q$, such that $t_1\leq t(q)\leq t_2$ and write $q^2-1 = p_1^{s_1} \cdots p_{t(q)}^{s_{t(q)}}$, where the $p_i$'s are the (distinct) prime factors of $q^2-1$ in ascending order. It is clear that $W(q^2-1) = 2^{t(q)}$, thus a condition, for our purposes, that is implied by Proposition~\ref{prop:siev1} is
\begin{equation} \label{eq:cond_alg}
q \geq \left( 2\cdot 2^{t(q) -s}\cdot \left( \frac{s-1}{\varepsilon} + 2 \right) \right)^2 .
\end{equation}
Now, notice that, $p_i\leq p(i)$, which means that $\varepsilon_1 \leq \varepsilon = 1-\sum_{i=0}^{s-1} 1/p_{t_1-i}$, and that $t(q)\leq t_2$, that is, the quantity $q_1$ computed in Step~2 is in fact larger than the RHS of \eqref{eq:cond_alg}, hence if $q\geq q_1$, then \eqref{eq:cond_alg} holds. The number $c$ in Step~3 stands for the maximum number of prime divisors a number not larger than $q_1^2-1$ can admit. This means that if $c\leq t_1\leq t(q)$, then, \eqref{eq:cond_alg} holds, and this is exactly the test that is performed in Step~4.

We successfully apply the algorithm for the pairs $(t_1,t_2)=(11,13)$ and $(10,10)$, which means that the case $t(q)\geq 10$ is settled. Thus, we may now assume that $t(q)\leq 9$ and we may now focus on the case
\[ q \leq (2\cdot 2^9)^2 = 1{,}048{,}576. \]

In the interval $3\leq q\leq 1{,}048{,}576$, there are exactly $82{,}247$ odd prime powers and we first attempt to use Proposition~\ref{prop:n=2(1)}. A quick computation reveals that, in the interval in question, there are exactly $2{,}425$ odd prime powers, where \eqref{eq:n=2_ip(1)} or \eqref{eq:n=2_ip(2)}, accordingly, do not hold, with all the mentioned quantities explicitly computed, with $q=1{,}044{,}889$ being the largest among them.

Then, we move on to the sieving part, i.e., Theorem~\ref{thm:siev4}. Namely, we attempt to satisfy the conditions of Theorem~\ref{thm:siev4} as follows. Until we run out of prime divisors of $k$, or until $\varepsilon\leq 0$, we always add to the set of sieving primes, that is, the primes $p_1,\ldots ,p_s$ in Theorem~\ref{thm:siev4}, the largest prime divisor not already in the set. If, for one such set of sieving primes, the condition of Theorem~\ref{thm:siev4} is valid, then
the desired result holds for the prime power in question.

This procedure was successful, for all the $2{,}425$ prime powers mentioned earlier, with the $101$ exceptions, corresponding to the values for $n=2$ in  Table~\ref{tab:exc}.
%

So, to sum up, we have proved the following.
\begin{theorem}\label{thm:main_n=2}
For every odd prime power $q$ not listed in Table~$\ref{tab:exc}$ and $\beta\in\F_{q}^*$, there exists some $\xi\in\F_{q^2}$ such that $\Tr(\xi)=\beta$.
\end{theorem}
\subsection{Completion of the proof of Theorem~\ref{thm:main}}
Then we move on to an explicit verification for the remaining possible exceptions, that is the pairs of Table~\ref{tab:exc}. For this purpose,  for all the corresponding pairs $(q,n)$, we check whether the set of the traces of the $2$-primitive elements of $\F_{q^n}$ coincides with $\F_q^*$, when $n=2$, and with $\F_q$, when $n\geq 3$. This test required about 5 minutes of computer time in a modern mid-range laptop.

\begin{table}[t] \footnotesize
\begin{center}
\begin{tabular}{|l|p{0.8\textwidth}|l|} \hline
$n$ & $q$ & \# \\ \hline\hline
$2$ & $3$, $5$, $7$, $9$, $11$, $13$, $17$, $19$, $23$, $25$, $27$, $29$, $31$, $37$, $41$, $43$, $47$, $49$, $53$, $59$, $61$, $67$, $71$, $73$, $79$, $81$, $83$, $89$, $97$, $101$, $103$, $109$, $113$, $121$, $125$, $127$, $131$, $137$, $139$, $149$, $151$, $157$, $169$, $173$, $181$, $191$, $197$, $199$, $211$, $229$, $239$, $241$, $269$, $281$, $307$, $311$, $331$, $337$, $349$, $361$, $373$, $379$, $389$, $409$, $419$, $421$, $461$, $463$, $509$, $521$, $529$, $569$, $571$, $601$, $617$, $631$, $659$, $661$, $701$, $761$, $769$, $841$, $859$, $881$, $911$, $1009$, $1021$, $1231$, $1289$, $1301$, $1331$, $1429$, $1609$, $1741$, $1849$, $1861$, $2029$, $2281$, $2311$, $2729$, $3541$ & 101 \\ \hline
$3$ & $5$, $9$, $13$, $25$, $37$, $49$, $121$ &  \;\; 7 \\ \hline\hline
\multicolumn{2}{|r|}{\textbf{Total:}} & 108 \\ \hline
\end{tabular}
\end{center}
\caption{Pairs $(q,n)$ for which the existence of $2$-primitive elements with prescribed trace was not dealt with theoretically.\label{tab:exc}}
\end{table}
The computations validated all the existence claims in Theorem~\ref{thm:main} for all the pairs $(q,n)$ of Table~\ref{tab:exc}  with the exception, when $n=2$, of $q=3,5,7,9,11,13$ and $31$,  these being genuine exceptions. In particular, they were successful  for all pairs $(q,n)$ with $n\geq 3$.  Finally, for the exceptions we present the possible traces of $2$-primitive elements in Table~\ref{tab:n=2_traces}. This completes the proof of Theorem~\ref{thm:main}.
\begin{table}[hbt]\footnotesize
\begin{center}
\begin{tabular}{|l|p{0.55\textwidth}|l|} \hline
$q$ & Traces & \# \\ \hline\hline
$3$ & $0$ & 1\\ \hline
$5$ & $2$, $3$ & 2\\ \hline
$7$ & $1$, $2$, $5$, $6$ & 4\\ \hline
$9^*$ & $\pm 1$, $ \pm i$ & 4 \\ \hline
$11$ & $1$, $2$, $3$, $4$, $7$, $8$, $9$, $10$ & 8 \\ \hline
$13$ & $1$, $3$, $4$, $5$, $6$, $7$, $8$, $9$, $10$, $12$ & 10 \\ \hline
$31$ & $1$, $2$, $3$, $4$, $5$, $6$, $7$, $8$, $9$, $10$, $12$, $13$, $14$, $15$, $16$, $17$, $18$, $19$, $21$, $22$, $23$, $24$, $25$, $26$, $27$, $28$, $29$, $30$ & 28 \\ \hline
\multicolumn{3}{l}{\rule{0em}{1.2em}\textbf{*} For $q=9$, $i$ is a root of $X^2+1 \in\F_3[X]$}
\end{tabular}
\end{center}
\caption{Traces of $2$-primitive elements of $\F_{q^2}$ for $q=3,5,7,9,11,13$ and $31$.\label{tab:n=2_traces}}
\end{table}
\subsection*{Acknowledgements}
The concept of this problem was realized during the second author's visit at Sabanc\i{} University and it came as a natural continuation of his joint work with Lucas Reis~\cite{kapetanakisreis18}.



\begin{bibdiv}
\begin{biblist}

\normalsize
\baselineskip=17pt




\bib{martinezreis16}{article}{
      author={Brochero Mart\'{i}nez, F. E.},
      author={Reis, Lucas},
       title={Elements of high order in {A}rtin-{S}chreier extensions of finite
  fields $\F_q$},
        date={2016},
     journal={Finite Fields Appl.},
      volume={41},
       pages={24\ndash 33},
}


\bib{cohen90}{article}{
      author={Cohen, Stephen~D.},
       title={Primitive elements and polynomials with arbitrary trace},
        date={1990},
     journal={Discrete Math.},
      volume={83},
             pages={1\ndash 7},
}
\bib{cohen05a}{article}{
      author={Cohen, Stephen~D.},
       title={Finite field elements with specified order and traces},
        date={2005},
     journal={Des. Codes Cryptogr.},
      volume={36},
      number={3},
       pages={331\ndash 340},
}

\bib{cohen10}{inproceedings}{
      author={Cohen, Stephen~D.},
       title={Primitive elements on lines in extensions of finite fields},
        date={2010},
   booktitle={Finite fields: Theory and applications},
      editor={McGuire, Gary},
      editor={Mullen, Gary~L.},
      editor={Panario, Daniel},
      editor={Shparlinski, Igor~E.},
      series={Contemp. Math.},
      volume={518},
   publisher={American Mathematical Society},
     address={Province, RI},
       pages={113\ndash 127},
}



\bib{cohenhuczynska03}{article}{
      author={Cohen, Stephen~D.},
      author={Huczynska, Sophie},
       title={The primitive normal basis theorem {{--}} without a computer},
        date={2003},
     journal={J. London Math. Soc.},
      volume={67},
      number={1},
       pages={41\ndash 56},
}

\bib{cohenkapetanakis20}{article}{
      author={Cohen, Stephen~D.},
      author={Kapetanakis, Giorgos},
       title={The trace of $2$-primitive elements of finite fields},
        date={2020},
     journal={Acta Arith.},
      volume={192},
       pages={397\ndash 416},
}


\bib{cohenpresern05}{article}{
      author={Cohen, Stephen~D.},
      author={Pre\u{s}ern, Mateja},
       title={Primitive finite field elements with prescribed trace},
        date={2005},
     journal={Southeast Asian Bull. Math.},
      volume={29},
      number={2},
       pages={383\ndash 300},
}

\bib{gao99}{article}{
      author={Gao, Shuhong},
       title={Elements of provable high orders in finite fields},
        date={1999},
     journal={Proc. Amer. Math. Soc.},
      volume={127},
      number={6},
       pages={1615\ndash 1623},
}

\bib{hua82}{book}{
      author={Hua, Loo Keng},
  translator={Shiu, Peter},
       title={Introduction to number theory},
   publisher={Springer-Verlag},
     address={Berlin Heidelberg},
        date={1982},
}

\bib{huczynskamullenpanariothomson13}{article}{
      author={Huczynska, Sophie},
      author={Mullen, Gary~L.},
      author={Panario, Daniel},
      author={Thomson, David},
       title={Existence and properties of $k$-normal elements over finite
  fields},
        date={2013},
     journal={Finite Fields Appl.},
      volume={24},
       pages={170\ndash 183},
}

\bib{kapetanakisreis18}{article}{
      author={Kapetanakis, Giorgos},
      author={Reis, Lucas},
       title={Variations of the primitive normal basis theorem},
        date={2019},
        journal={Des. Codes Cryptogr.},
        volume={87},
        number={7},
        pages={1459\ndash 1480}
}

\bib{katz89}{article}{
      author={Katz, Nicholas~M.},
       title={An estimate for character sums},
        date={1989},
     journal={J. Amer. Math. Soc.},
      volume={2},
      number={2},
       pages={197\ndash 200},
}

\bib{lidlniederreiter97}{book}{
	author={Lidl,Rudolf},
	author ={Niederreiter, Harald},
	title={Finite Fields},
	series={Encylopedia of Mathematics and its Applications},
	publisher={Cambridge University Press} 
	addess={Cambridge}
	date={1997}

}

\bib{popovych13}{article}{
      author={Popovych, Roman},
       title={Elements of high order in finite fields of the form
  ${F}_q[x]/(x^m-a)$},
        date={2013},
     journal={Finite Fields Appl.},
      volume={19},
      number={1},
       pages={96\ndash 92},
}

\bib{schmidt76}{book}{
      author={Schmidt, Wolfgang~M.},
       title={Equations over finite fields, an elementary approach},
   publisher={Springer-Verlag},
     address={Berlin Heidelberg},
        date={1976},
}

\bib{wenpeng02}{article}{
      author={Zhang, Wenpeng},
       title={Moments of generalized quadratic Gauss sums weighted by $L$-functions},
        date={2002},
     journal={J. Number Theory},
      volume={92},
      number={2},
       pages={304\ndash 314},
}

\end{biblist}
\end{bibdiv}
\end{document}